\documentclass[10pt,reqno]{amsart}
\usepackage{url}
\usepackage{amsthm}
\usepackage{fullpage}
\usepackage{amsmath}
\usepackage{amssymb}
\usepackage{amsfonts}

\usepackage{enumitem}
\usepackage{mathscinet}
\usepackage{lipsum}
\usepackage[english]{babel}
\usepackage[autostyle]{csquotes}
\usepackage{bbold}
\usepackage{hyperref}

\hypersetup{
    colorlinks=true,
    linkcolor=red,
    citecolor=blue,
    }

\usepackage{graphicx}

\usepackage[utf8]{inputenc}
\usepackage[english]{babel}

\usepackage{color}
\usepackage{comment}

\newtheorem{thm}{Theorem}[section]
\newtheorem{definition}[thm]{Definition}
\newtheorem{theorem}[thm]{Theorem}
\newtheorem*{oseledec*}{Oseledec Theorem}
\newtheorem{cor}[thm]{Corollary}

\newtheorem{lemma}[thm]{Lemma}

\makeatletter
\def\moverlay{\mathpalette\mov@rlay}
\def\mov@rlay#1#2{\leavevmode\vtop{%
   \baselineskip\z@skip \lineskiplimit-\maxdimen
   \ialign{\hfil$\m@th#1##$\hfil\cr#2\crcr}}}
\newcommand{\charfusion}[3][\mathord]{
    #1{\ifx#1\mathop\vphantom{#2}\fi
        \mathpalette\mov@rlay{#2\cr#3}
      }
    \ifx#1\mathop\expandafter\displaylimits\fi}
\makeatother

\newcommand{\nocontentsline}[3]{}
\newcommand{\tocless}[2]{\bgroup\let\addcontentsline=\nocontentsline#1{#2}\egroup}

\usepackage{wasysym}

\def\Vol{\ensuremath{\mathrm{Vol}}}

\def\E{\ensuremath{\mathcal{E}}}
\def\Et{\ensuremath{\mathcal{E}_{\mathrm{top}}}}
\def\h{\ensuremath{h^\mathrm{BK}}}

\title{Neutralized Local Entropy and dimension bounds for invariant measures}

\author{S. Ben Ovadia$^{\star}$, F. Rodriguez-Hertz$^{\dagger}$}

\thanks{$^{\star}$ Department of Mathematics, Eberly College of Science, Pennsylvania State University, snir.benovadia@psu.edu \\ $\text{ }$$\text{ }$ $\text{ }$ $^\dagger$ Department of Mathematics, Eberly College of Science, Pennsylvania State University, fjr11@psu.edu}

\begin{document}
\maketitle
\begin{abstract}
We introduce a notion of a point-wise entropy of measures (i.e local entropy) called {\em neutralized local entropy}, and compare it with the Brin-Katok local entropy. We show that the neutralized local entropy coincides with Brin-Katok local entropy almost everywhere. Neutralized local entropy is computed by measuring open sets with a relatively simple geometric description. Our proof uses a measure density lemma for Bowen balls, and a version of a Besicovitch covering lemma for Bowen balls. As an application, we prove a lower point-wise dimension bound for invariant measures, complementing the previously established bounds for upper point-wise dimension.% In particular this extends a previous result for hyperbolic measures, by a simplified proof.
\end{abstract}

\section{Introduction and main results}\label{intro}
One of the most useful tools in studying dynamical systems are dynamical covers and partitions of a measure. That is, given an invariant measure, a cover or a partition modulo the measure, whose elements have a significant structure w.r.t the dynamics. The pursuit for such objects creates a tension between searching for elements with a simple geometric description, with which it is easy to work, and searching for elements with significant dynamical structure, which allows to control their orbits and measure. Two charecteristic examples of this tension are geometric balls, and Bowen balls. 

\medskip
Geometric balls allow one to utilize strong geometric properties such as the Lebesgue density theorem, or the Besicovitch covering lemma. Bowen balls on the other hand, allow one to estimate their measure in terms of the entropy (see \cite{BK}), and control their image under the dynamics for a fixed amount of iterations; while having a possibly very complicated geometric shape. 

\medskip
Pesin theory allows one to linearize locally the action of the dynamics on typical orbits. However, the size of the neighborhood where the linearization is valid may deteriorate along the orbit, although in a sub-exponential rate. Given an ergodic measure with positive metric entropy, the Ruelle inequality says it must have some positive and some negative Lyapunov exponents (\cite{MargulisRuelleIneq}). When an orbit admits a central direction, i.e an invariant sub-space of the tangent space corresponding to $0$ Lyapunov exponents, the central direction may not be integrable into an invariant manifold. In addition, the differential may contract or expand tangent vectors in the central direction, in a sub-exponential way. These effects make it very hard, and generally not attainable, to get a simple description of the set of points which remain close to the orbit for a fixed amount of steps.

\medskip
Furthermore, even in the absence of $0$ Lyapunov exponents, the decay of the size of the Pesin chart (i.e neighborhood with a local linearization of the dynamics) does not allow to control the set of points which remain close to the orbit by a fixed distance (i.e Bowen balls). Hence, phasing-out sub-exponential effects such as the central direction or the deterioration of the size of Pesin charts becomes very useful. Removing these effects allows one to treat the action of the dynamics along an orbit as if they were linear and hyperbolic, and so simplifying greatly the geometric description of the set of points which shadow the orbit (or part of it).  

\medskip
Our goal in this paper is to address exactly this difficulty. Our proof relies on a sub-exponential measure density lemma over Bowen balls (rather than geometric balls as in the Lebesgue density lemma).

\medskip
In our setup $M$ is a closed Riemannian manifold, $d=\mathrm{dim}M\geq2$, and $f\in \mathrm{Diff}^{1+\beta}(M)$, $\beta>0$. Let $\mu$ be an $f$-invariant Borel probability measure. The purpose of this paper is to compute the {\em neutralized local entropy} defined by,
\begin{equation}\label{defNeutLocEnt}\mathcal{E}_\mu(x):=\lim_{r\to0}\limsup_{n\to\infty}\frac{-1}{n}\log \mu(B(x,n,e^{-rn})),\end{equation}

where $B(x,n,e^{-rn}):=\{y\in M: d(f^{i}(y),f^{i}(y)(x))\leq e^{-rn}, \forall 0\leq i\leq n\}$. One can check the following properties for the neutralized local entropy% (in our smooth setup)
:
\begin{enumerate}
\item[(a)] \begin{equation}\label{NeuLocEntBdd}\mathcal{E}_\mu(x)\in [h_\mu^\mathrm{BK}(x),d\cdot \log M_f]\text{ }\mu\text{-a.e,}\end{equation}
where $M_f:=\max_{x\in M}\{\|d_xf\|, \|d_xf^{-1}\|\}$, and $h_\mu^\mathrm{BK}(x)$ is the local entropy at $x$ given by the Brin-Katok formula, \cite{BK}.\footnote{$h_\mu^\mathrm{BK}(x)=\lim_{\epsilon\to0}\limsup\frac{-1}{n}\log\mu(B(x,n,\epsilon))= \lim_{\epsilon\to0}\liminf\frac{-1}{n}\log\mu(B(x,n,\epsilon))$, where the limits exist and the inequality holds $\mu$-a.e, and $B(x,n,\epsilon)=\{y\in M: d(f^{i}(x),f^{i}(y))\leq\epsilon,\forall 0\leq i\leq n\}$.} The upper bound is given by the classical Lemma which states that $\lim_{s\to0}\frac{\log\mu(B(x,s))}{\log s}\leq d$ for $\mu$-a.e $x$% (see, for example, \cite[Lemma~4.1.4]{LedrappierYoungI})
, and by the fact that $B(x,n,e^{-rn})\supseteq B(x,M_f^{-n})$ for all $n\geq0$ and sufficiently small $r>0$.
\item[(b)] 	\begin{equation}\label{NeutLocEntInv}\mathcal{E}_\mu\circ f= \mathcal{E}_\mu\text{ }\mu\text{-a.e.}\end{equation}
This can be seen by the following two inequalities:
\begin{enumerate}
\item[(I)] $\mathcal{E}_\mu(f(x))\geq\mathcal{E}_\mu(x)$: since $$\mu(B(f^{-1}(x),n,e^{-rn}))=\mu(f[B(f^{-1}(x),n,e^{-rn})])\leq \mu(B(x,n,M_fe^{-rn})),$$
\item[(II)] $\mathcal{E}_\mu(f(x))\leq\mathcal{E}_\mu(x)$: since $$\mu(B(f^{-1}(x),n,e^{-rn}))\geq \mu(f^{-1}[B(x,n+1,e^{-r\frac{n}{n+1}\cdot(n+1)})])\geq \mu(B(x,n+1,e^{-r(n+1)})).$$
\end{enumerate}
\end{enumerate}
\begin{comment}
In particular, if $\mu$ is ergodic, then the neutralized local entropy of $\mu$ is constant almost everywhere, and dominates the metric entropy of $\mu$ (and lies within a globally-bounded interval).
\end{comment}

\medskip
The significance of the neutralized entropy is that it estimates the asymptotic measure of sets with a distinctive geometric shape. Unlike the sequence $\{B(x,n,r)\}_{n\geq0}$ which can develop a very complicated geometric shape for large $n$- due to a central direction, or even for a non-uniformly hyperbolic trajectory- the sequence $\{B(x,n,e^{-rn})\}_{n\geq0}$ can have a nice description a.e for any $r>0$, by neutralizing any sub-exponential effects. Sets with a more explicit geometric description are very useful for the construction of covers (or consequently even partitions), and so controlling the measure of such sets is important. 

\medskip
In fact, one can guess that by the lack of diversity for intrinsic dynamical invariants, the neutralized local entropy must coincide with other notions of local entropy (or the metric entropy in the ergodic case). This paper is dedicated to the proof of this statement for smooth systems. 

Thieullen \cite{Thieullen92} studied a similar notion to the neutralized local entropy, called $\alpha$-entropy, for certain systems on infinite-dimensional systems (see also \cite{Thieullen92b,Thieullen90}). Other generalizations of entropy have also been studied in terms of ergodic theory for some systems by \cite{TakensVerbitskiy2002,TakensVerbitskiy2003}.

\medskip
As an application of the neutralized local entropy, in Theorem \ref{LowDimBound} we prove a lower bound for the point-wise dimension of ergodic invariant measures: for almost every point
$$\liminf_{r\to 0}\frac{\log\mu(B(x,r))}{\log}\geq d^u+d^s,$$
where $d^u$ and $d^s$ are the point-wise dimension of the conditional measures of $\mu$ along the unstable and the stable laminations, resp. (see \textsection \ref{LowDim} for details).	This extends the previous result in \cite{AsymProd} for hyperbolic measures, and complements the upper bound in \cite{LedrappierYoungII}. In particular, using the notion of neutralized local entropy, the proof of Theorem \ref{LowDimBound} is relatively elementary as it only uses covers and does not require constructing adapted partitions. 

\section{Covering, differentation, and Ergodic theorems}\label{lemmas}

Let $M$ be a closed Riemannian manifold, $d=\mathrm{dim}M\geq2$, and let $f\in \mathrm{Diff}^{1+\beta}(M)$, $\beta>0$.

\begin{definition}
	Let $\mu$ be an $f$-invariant ergodic probability measure s.t $\chi(\mu)>0$. Denote by $(\chi_1,m_1,\ldots ,\chi_{\ell_{\underline{\chi}}},m_{\ell_{\underline{\chi}}})= \underline{\chi}(\mu)$ the Lyapunov exponents and dimension of $\mu$ in a decreasing order. 
\begin{enumerate}
	\item Let $0<\tau\leq \tau_{\underline{\chi}}:= \frac{1}{100d}\min\{\chi_{\ell_{\underline{\chi}}},\chi_{i+1}-\chi_i: i\leq \ell_{\underline{\chi}}-1\}$, and let $C_{\underline{\chi},\tau}(\cdot)$ be the Lyapunov change of coordinates for points in $\mathrm{LR}_{\underline{\chi}}=\{\text{Lyapunov regular points with an index }\tau_{\underline\chi}\}$ (see \cite{KM}).
	\item Let $\mathrm{PR}_{\underline{\chi}}=\{x\in\mathrm{LR}_{\underline{\chi}}:\limsup_{n\to\pm\infty}\frac{1}{n}\log\|C_{\underline{\chi},\tau}^{-1}(f^n(x))\|=0,\forall 0<\tau\leq \tau_ {\underline{\chi}}\}$, the set of {\em $\underline{\chi}$-Pesin regular} points which carries $\mu$. $\mathrm{PR}:=\bigcup_{\underline\chi}\mathrm{PR}_{\underline\chi}$ is called the set of {\em Pesin regular} points.
\item Given $x\in \mathrm{PR}_{\underline{\chi}}$, let $E_j(x)$ be the Oseledec subspace of $x$ corresponding to $\chi_j$.
\item Let $\widehat{E}_j:=C^{-1}_{\underline{\chi},\tau}(x)[E_j(x)]$, and notice that $j\neq j'\Rightarrow \widehat{E}_j \perp \widehat{E}_{j'}$. Write $\widehat{E}^\mathrm{cs}:=\oplus_{\chi_j\leq 0} \widehat{E}_j$. 
%\item Let $\mathrm{PR}_{\underline{\chi}}'=\{x\in\mathrm{PR}_{\underline{\chi}}:\lim_{n\to\infty}\frac{1}{n}\log\|d_0\exp^{W_i}_{f^n(x)}\|= \lim_{n\to\infty}\frac{1}{n}\log\|d_0\exp^{W_i}_{f^n(x)}\|_{\mathrm{co}} =0\}$, which carries $\mu$.
\item A {\em Pesin block} $\Lambda^{(\underline{\chi},\tau)}_\ell$ is a subset of $\bigcup_{|\underline{\chi}'-\underline\chi|_\infty< \tau}\mathrm{PR}_{\underline\chi'}$ which is a level set $[q_\tau\geq\frac{1}{l}]$ of a measurable function $q_\tau: \bigcup_{|\underline{\chi}'-\underline\chi|_\infty< \tau}\mathrm{PR}_{\underline\chi'}\to (0,1)$ s.t (a) $\frac{q_\tau\circ f}{q_\tau}=e^{\pm \tau}$, (b)  $q_\tau(\cdot)\leq\frac{1}{\|C^{-1}_{\underline\chi,\tau}(\cdot)\|^\frac{d}{\beta}}$%, and (c) $q_\tau|_{\Lambda^{(\underline\chi,\tau)}}$ is bounded
. Often we omit the subscript $\ell$ when the dependence on $\ell$ is clear from the context.
\end{enumerate}	
\end{definition}
%By convention, if all exponents of $\mu$ are non-positive, $\tau_{\underline\chi(\mu)}:=\infty$.

%Let $i\leq\ell_{\underline{\chi}}$.
\begin{lemma}[Besicovitch-Bowen covering lemma]\label{BBCL}
	Let $\Lambda^{(\underline{\chi},\tau)}$ ($0<\tau\leq\tau_{\underline{\chi}}$) be a Pesin block, and let $x_0\in\Lambda^{(\underline{\chi},\tau)}$. Let $B(x_0)$ be the Pesin chart of $x_0$ for $\Lambda^{(\underline{\chi},\tau)} $. Let $A\subseteq \Lambda^{(\underline{\chi},\tau)} \cap B(x_0)$ be a measurable subset. Then $A$ can be covered by a %finite 
	cover of exponential Bowen balls of points in $A$ (i.e $B(\cdot,n,e^{-n\epsilon})$), with multiplicity bounded by $e^{3d\tau n}$, where $n$ is sufficiently large w.r.t $\Lambda^{(\underline{\chi},\tau)}$, and $\epsilon\geq 2\tau$.
\end{lemma}
The idea of the proof follows the principle steps of the proof of the classical Besicovitch covering lemma (see \cite[\textsection~18]{RealAnalysisBook}), using a certain volume doubling property for exponential Bowen balls.
\begin{proof}
Let $n\in\mathbb{N}$ sufficiently large so $e^{-\epsilon n}$ is smaller than the Pesin chart size for $\Lambda^{(\underline{\chi},\tau)}$. Set $A_1:=A$, and choose $x_1\in A_1$. Given $k$, set $A_{k+1}:=A\setminus \bigcup_{j=1}^kB(x_j,n,e^{-n\epsilon})$ and choose $x_{k+1}\in A_{k+1}$. Continue in this process as long as $A\setminus \bigcup_{j=1}^kB(x_j,n,e^{-n\epsilon})$ is not empty. 

\medskip
Claim 1: $\forall k< j$,  $B(x_j,n,\frac{1}{3}e^{-n\epsilon})\cap B(x_k,n,\frac{1}{3}e^{-n\epsilon}) =\varnothing$.

Proof: Otherwise, $\forall i\leq n$, $d(f^i(x_k), f^i(x_j))\leq \frac{2}{3}e^{-\epsilon n}$, whence $x_j\in B(x_k,n,e^{-n\epsilon})$, in contradiction to the choice of $\{x_l\}_{l\geq1}$.

\medskip
Claim 2: $\exists N\in\mathbb{N}$ s.t $\bigcup^N B(x_k,n,e^{-\epsilon n})\supseteq A$.

Proof: For any $k\geq 1$, $\Vol(B(x_k,n,\frac{1}{3}e^{-\epsilon n}))= C^{\pm1}\cdot e^{-\underline{\chi}^u n}\cdot e^{-d_\mathrm{cs}\epsilon n}e^{\pm d\tau n}$, where $e^{-\underline{\chi}^u n}=\prod_{i\leq \ell_{\underline{\chi}}: \chi_i>0}e^{-\chi_i n\cdot d_i}$, $d_i$ is the multiplicity of $\chi_i$, $d_\mathrm{cs}:=\sum_{i:\chi_i>0}d_i$, and $C$ is a constant depending on $M$, $d$, and $\Lambda^{(\underline\chi,\epsilon)}$. Then $N\leq \frac{\Vol(M)}{C^{-1}\cdot e^{-\underline{\chi}^un}\cdot e^{-d\tau n}\cdot e^{-d_\mathrm{cs}\epsilon n}}$.

\medskip
Claim 3: $\bigcup_{k=1}^NA_k\supseteq A$.

Proof: The process continues unless $A$ is covered.

\medskip
Claim 4: $\forall k\leq N$, $\#\{j\leq N: B(x_j,n,e^{-n\epsilon})\cap B(x_k,n,e^{-n\epsilon})\neq\varnothing \}\leq C_d e^{2dn\tau}$, where $C_d$ is a constant depending on $M$ and on $\Lambda^{(\underline\chi,\tau)}$. 

Proof: Let $k\leq N$, and let $k\neq j\leq N$ s.t $B(x_j,n,e^{-n\epsilon})\cap B(x_k,n,e^{-n\epsilon})\neq\varnothing$. Then
$$B(x_j,n,\frac{1}{3}e^{-n\epsilon})\subseteq B(x_j,n,e^{-n\epsilon}) \subseteq B(x_k,n,3e^{-n\epsilon}).$$
Then 
$$\#\{j\leq N: B(x_j,n,e^{-n\epsilon})\cap B(x_k,n,e^{-n\epsilon})\neq\varnothing \}\leq \frac{\max_{k\leq N}\Vol(B(x_k,n,3e^{-n\epsilon}))}{\min_{j\leq N} \Vol(B(x_j,n,\frac{1}{3}e^{-n\epsilon}))}\leq C^2 3^{2d}e^{2d\tau n}.$$

\medskip
Claim 5: We can divide $\{x_k\}_{k\leq N}$ into sub-collections $\mathcal{C}_i$, $i=1,\ldots, \lceil C_de^{2nd\tau}+1\rceil $, where for any $i\leq \lceil C_de^{2dn\tau}+1\rceil$, $\{B(x,n,e^{-\epsilon n}):x\in \mathcal{C}_i\}$ is a mutually disjoint collection. In particular, for $n$ sufficiently large, $\lceil C_de^{2n\tau}+1\rceil\leq e^{3dn\tau}$. 

Proof: Let $K_n:=\lceil C_d e^{2nd\tau}+1\rceil$. We associate $x_k$ with $\mathcal{C}_k$ for all $k\leq K_n$. For each $k> K_n$, we allocate it into one of the pre-existing collections in the following way. 

Given $B(x_{K_n+1},n,e^{-\epsilon n})$, by claim 4, there exists at least one $i_{K_n+1}\leq K_n$ s.t $B(x_{K_n+1},n,e^{-\epsilon n})\cap B(x_{i_{K_n+1}},n,e^{-\epsilon n}) =\varnothing$; so allocate $x_{K_n+1}$ to $\mathcal{C}_{i_{K_n+1}}$.

Next, consider $B(x_{K_n+2},n,e^{-\epsilon n})$, at least two of the $K_n+1$-first balls do not intersect it. If any of the balls $2,\ldots,K_n$ do not intersect $B(x_{K_n+2},n,e^{-\epsilon n})$, allocate it to the first associated collection as such. If all of the balls of $1,\ldots,i_{K_n+1}-1, i_{K_n+1}+1,\ldots, K_n$ intersect it, then allocate it to $\mathcal{C}_{i_{K_n+1}}$ which now contains three disjoint balls. 

We continue by induction. Assume that the balls $\{B(x_j,n,e^{-\epsilon n})\}{j\leq K_n+l-1}$ have all been allocated into one of the $K_n$-many disjoint sub-collections. Consider $B(x_{K_n+l},n,e^{-\epsilon n})$, which is disjoint from at least $l$ balls in $\{B(x_j,n,e^{-\epsilon n})\}{j\leq K_n+l-1}$. This implies that least  one of the collections $\mathcal{C}_i$, $i\leq K_n$, is disjoint from $B(x_{K_n+l},n,e^{-\epsilon n})$, to which we may allocate it.
\end{proof}

\begin{lemma}[Bowen-Lebesgue density lemma]
\label{LLYD}
Let $\mu$ be an $f$-invariant probability measure. Let $A$ be a measurable set s.t $\mu(A)>0$. Then for $\mu$-a.e $x\in A$, 
$$\lim_{r\to0}\limsup_{n\to\infty}\frac{-1}{n}\log\frac{\mu(B(x,n,e^{-nr})\cap A)}{\mu(B(x,n,e^{-nr}))}=0.$$ 
\end{lemma}
\begin{proof}
First, assume that $\exists \lambda>0$ s.t (w.l.o.g) for $\mu$-a.e $x\in A$,
$$\lim_{r\to0}\limsup\frac{-1}{n}\log\frac{\mu(B(x,n,e^{-r n})\cap A)}{\mu(B(x,n,e^{-r n}))}\geq \lambda.$$

Let $0<r<\frac{1}{3d}\frac{\lambda}{4}$ s.t $\mu(A^{(0)})>0$ where $$A^{(0)}:=\{x\in A: \limsup\frac{-1}{n}\log\frac{\mu(B(x,n,e^{-r n})\cap A)}{\mu(B(x,n,e^{-r n}))}\geq\frac{7\lambda}{8}\}.$$

Then it follows that $\mu(A^{(1)})>0$ where $A^{(1)}:=A^{(0)}\cap \Lambda^{(\underline\chi,\tau)}$ for some index $\underline\chi$ and $\tau\leq\min\{\frac{r}{2},\tau_{\underline\chi}\}$ (this can be achieved by first dividing the parameter space of $\underline\chi$ by the Oseledec dimensions, then by boxes around each $\underline\chi'$ of size $\frac{1}{2}\tau_{\underline{\chi}'}$, and then further by boxes of size $\tau:=\min\{\frac{r}{2},\frac{1}{2}\tau_{\underline\chi'}\}$; whence $\frac{1}{2}\tau_{\underline\chi'}\leq \tau_{\underline\chi}$).

\medskip
Let $N\geq1$ large and $x\in A ^{(1)} $, and set
\begin{align*}
n_x^N:=\min\{n\geq N: \frac{-1}{n}\log\frac{\mu(B(x,n,e^{-r n})\cap A)}{\mu(B(x,n,e^{-r n}))}\geq\frac{6\lambda}{8}\}.
\end{align*}

Set for all $n\geq N$,
$$A_n ^{(1)}:=\{x\in A ^{(1)}: n_x^N=n\}.$$

\medskip
By Lemma \ref{BBCL}, we can choose a finite subset $\mathcal{C}_n\subseteq A^{(1)}_n $ s.t $\bigcup_{x\in \mathcal{C}_n}B(x,n,e^{-rn})\supseteq A^{(1)}_n$ while $\{B(x,n,e^{-r n})\}_{x\in\mathcal{C}_n}$ has an overlap bound smaller than $e^{3d\tau n}\leq e^{3d\frac{r}{2}n}\leq e^{\frac{\lambda}{8}n}$ (for all $n$ large enough depending on $M$ and $\Lambda^{(\underline\chi,\tau)}$). Then,
\begin{align*}
	\mu(A ^{(1)}_n)= & \mu\left((\bigcup_{x\in\mathcal{C}_n}B(x,n,e^{-rn}))\cap A^{(1)}_n\right) \leq \sum_{x\in \mathcal{C}_n}\mu(B(x,n,e^{-rn})\cap A ^{(1)}_n) \leq 
\sum_{x\in \mathcal{C}_n}\mu(B(x,n,e^{-rn})\cap A )\\
 \leq &\sum_{x\in \mathcal{C}_n}\mu(B(x,n,e^{-rn}))e^{-\frac{6\lambda}{8}n}
\leq e^{-\frac{6\lambda}{8}n} e^{\frac{\lambda}{8}n}\mu(\bigcup_{x\in\mathcal{C}_n}B(x,n,e^{-rn}))\leq e^{-\frac{5\lambda}{8}n}.
\end{align*} 

Then, 
$$\mu(A ^{(1)})=\sum_{n\geq N}\mu(A ^{(1)}_n)\leq \sum_{n\geq N} e^{-\frac{5\lambda}{8} n} \leq (\sum_{n\geq0}e^{-\frac{5\lambda}{8} n})\cdot  e^{-\frac{5\lambda}{8} N}\xrightarrow[]{N\to\infty}0,\text{ a contradiction!}$$
\end{proof}

\noindent\textbf{Remark:} A possible heuristic way to interpret Lemma \ref{LLYD} is the following: we think of belonging to a measurable set as satisfying some property. Then, for almost any point which satisfies a certain property, more and more points which spend a long portion of their orbit close to this point inherit this property as well. That is, a portion of the exponential Bowen ball, with bigger and bigger exponential portion bounds, lies in the measurable set as well.

\begin{cor}[log-differentation lemma]\label{BLDT}
Let $\mu$ be an $f$-invariant probability measure, and let $g\in\mathcal{M}(\mu)$ be a measurable function and $A\in\mathcal{B}$ be a measurable set with $\mu(A)>0$. Then for $\mu$-a.e $x\in A$,
$$\lim_{\epsilon\to0}\limsup_{n\to\infty}\frac{-1}{n}\log\left(\frac{1}{\mu(B(x,n,e^{-n\epsilon}))}\int_{B(x,n,e^{-n\epsilon})}e^{-n|g(y)-g(x)|}\mathbb{1}_A(y)d\mu(y)\right)=0.$$
\end{cor}
\begin{proof}
Let $\delta>0$, and let $a\in \mathbb{R}$ s.t $\mu(E)>0$ where $E:=A\cap g^{-1}[[a-\frac{\delta}{2},a+\frac{\delta}{2}]]$. Let $x\in E$, $n\geq 0$, and $\epsilon>0$, then 
\begin{align*}
	\frac{1}{\mu(B(x,n,e^{-\epsilon n}))}\int_{B(x,n,e^{-\epsilon n})}e^{-n|g(y)-g(x)|}\mathbb{1}_A(y)d\mu(y)\geq& \frac{1}{\mu(B(x,n,e^{-\epsilon n}))}\int_{B(x,n,e^{-\epsilon n})\cap E}e^{-n\delta }d\mu(y)\\
	=& e^{-n\delta }\frac{\mu(B(x,n,e^{-\epsilon n})\cap E)}{\mu(B(x,n,e^{-\epsilon n}))}.
\end{align*}
Then by Lemma \ref{BBCL}, for $\mu$-a.e $x\in E$, $\lim\limits_{\epsilon\to0}\limsup\limits_{n\to\infty}\frac{-1}{n}\log\frac{1}{\mu(B(x,n,e^{-n\epsilon}))}\int\limits_{B(x,n,e^{-n\epsilon})}e^{-n|g(y)-g(x)|}\mathbb{1}_A(y)d\mu(y)\leq \delta$. Since $\delta >0$ was arbitrary (and the limit is independent of $a$), we are done. 
\end{proof}

\begin{theorem}[Log-Ergodic Theorem]\label{LET}
Let $\mu$ be an $f$-invariant probability measure, let $A$ be a measurable set s.t $\mu(A)>0$, and let $g%:M\to\mathbb{R}^{\geq0}
\in L^1(\mu)$% be a non-negative function
. Then for $\mu$-a.e $x\in A$,
$$\lim_{r\to0}\limsup_{n\to\infty}\frac{-1}{n}\log\frac{1}{\mu(B(x,n,e^{-nr}))}\int_{B(x,n,e^{-nr})}\mathbb{1}_A(y)\cdot e^{-|\sum_{j=0}^{n-1}g\circ f^{-j}(y)-\sum_{j=0}^{n-1}g\circ f^{-j}(x)|}d\mu(y)=0.$$
\end{theorem}
\begin{proof}
Let $\mu=\int \mu_x d\mu(x)$ be the ergodic decomposition of $\mu$. Since $g\in L^1(\mu)$, for $\mu$-a.e $x$, $g\in L^1(\mu_x)$. Then for $\mu$-a.e $x$, $\lim_{n\to\infty}\frac{1}{n}\sum_{j=0}^{n-1}(g-\int gd\mu_x)\circ f^{-j}(x)=0$. Let $\delta>0$, and let $n_\delta\geq0$ s.t $\mu(A_\delta)\geq e^{-\delta}\mu(A)$ where $$A_\delta:=\left\{y\in A: \forall n\geq n_\delta, \left|\sum_{j=0}^{n-1}(g-\int gd\mu_y)\circ f^{-j}(y)\right|\leq n\delta\right\}.$$
Let $x\in A_\delta$, then for all $n\geq n_\delta$ and $r>0$,
\begin{align*}
\int\limits_{B(x,n,e^{-nr})}&\mathbb{1}_A(y)\cdot e^{-|\sum_{j=0}^{n-1}g\circ f^{-j}(y)-\sum_{j=0}^{n-1}g\circ f^{-j}(x)|}d\mu(y)\\
&\geq \int\limits_{B(x,n,e^{-nr})}\mathbb{1}_{A_\delta}(y)\cdot e^{-|\sum_{j=0}^{n-1} (g-\int gd\mu_y)\circ f^{-j}(y)|-n| \int gd\mu_y-\int gd\mu_x|-|\sum_{j=0}^{n-1} (g-\int gd\mu_x)\circ f^{-j}(x) |}d\mu(y)\\
&\geq e^{-2\delta n} \int_{B(x,n,e^{-nr})}\mathbb{1}_{A_\delta}(y)\cdot e^{-n| G(y)-G(x)|}d\mu(y),
\end{align*}
where $G(y):=\int gd\mu_y$. Then by the log-differentiation lemma (Corollary \ref{BLDT}) for $G$, and since $\delta>0$ was arbitrary, we are done.
\end{proof}

\section{Neutralized local entropy is entropy}\label{EProps0}

\begin{lemma}\label{forReduc}
Let $\mu$ be an $f$-invariant Borel probability, and let $\mu=\int \mu_x d\mu(x)$ be its ergodic decomposition. Then $\h_\mu(x)= h_{\mu_x}(f)$ $\mu$-a.e. %Assume that $\exists h,sigma>0$ s.t for $\mu$-a.e $x$, $h_{\mu_x}(f)=h+\pm\sigma$. Then $\E_\mu=h\pm \sigma$ $\mu$-a.e.
\end{lemma}
\begin{proof}
	Let $E_\lambda:=\{x: \h_\mu(x)\geq h_{\mu_x}(f)+\lambda\}$, and assume that there exists $\lambda>0$ s.t $\mu(E_\lambda)>0$ (notice that $E_\lambda$ is $f$-invariant). Let $G_\lambda:=\{x: \mu_x(E_\lambda)=1\}$ with $a_\lambda:=\mu(G_\lambda)>0$ (o.w $\mu(E_\lambda)=0$). Write $\mu_\lambda:=\frac{1}{\mu(G_\lambda)}\int_{G_\lambda}\mu_x d\mu(x)$, and $\mu_c:=\frac{1}{a_c}(1-\mu_\lambda)$ where $a_c:=1-a_\lambda$. Then $\mu=a_\lambda \mu_\lambda+a_c\mu_c$ and for $\mu_\lambda$-a.e $x$,%{\color{red} needs to be explained why $\mu_\lambda$ is carried by $E_\lambda$ (notice $E_\lambda$ is invariant)}
	
$$\h_{\mu_\lambda}(x)\geq \h_\mu(x)\geq h_{\mu_x}(f)+\lambda.$$
However, this is a contradiction, since integrating both sides by $\mu_\lambda$ admits $h_{\mu_\lambda}(f)\geq h_{\mu_\lambda}(f)+\lambda$, a contradiction! Hence $\h_\mu(x)\leq h_{\mu_x}(f)$ $\mu$-a.e, but $\int \h_\mu(x)d\mu(x)=\int h_{\mu_x}(f)d\mu(x)$, hence $\h_\mu(x)=h_{\mu_x}(f)$ $\mu$-a.e.	
\end{proof}

\begin{lemma}\label{hBddBowenBallsCover}
 Let $\mu$ be an $f$-invariant Borel probability, and $K$ be a $\mu$-positive measure set% measurable subset of $\Lambda^{(\underline\chi,\tau)}$ s.t $\mu(K)>0$
 , and let $\delta>0$. Assume that $\exists \Delta>0$ s.t $\h_\mu\leq h+\Delta$ for $\mu$-a.e $x\in K$. Then for all $n$ large enough w.r.t $\delta$ and $K$, there exist a measurable subset $K_\delta\subseteq K$ and a subset $\mathcal{A}_{n,\delta}$ and $0<\rho\leq \delta$ s.t 
	\begin{enumerate}
		\item	$\bigcup_{x\in\mathcal{A}_{n,\delta}} B(x,n,\rho)\supseteq K_\delta$,
		\item 	$\frac{\mu(K_{\delta})}{\mu(K)}\geq e^{-\delta}$,
		\item $\#\mathcal{A}_{n,\delta}\leq e^{n(h+\Delta+\delta)}$.
	\end{enumerate}
\end{lemma}
\begin{proof}
	For $\mu$-a.e $x\in K$, $\lim_{r\to0}\limsup\frac{-1}{n}\log\mu(B(x,n,r))\leq h+\Delta$. Set $K_\delta\subseteq K$ and $n_\delta\in\mathbb{N}$ s.t $\forall x\in K_\delta$, $\forall n\geq n_\delta$, $\mu(B(x,n,\frac{\rho}{3}))\geq e^{-n(h+\Delta+\delta)}$, and $\mu(K_\delta)\geq e^{-\delta}\mu(K)$, for some $0<\rho\leq\delta$.

Let $n\geq n_\delta$. Set $K^1:=K_\delta$, and let $x_1\in K^1$. $K^{i+1}:=K^i\setminus B(x_i,n,\rho)$, and choose $x_{i+1}\in K^{i+1}$. 

For any $i\neq j$, $B(x_i,n,\frac{\rho}{3})\cap B(x_j,n,\frac{\rho}{3})=\varnothing$, hence we have at most $e^{n(h+\Delta+\delta)}$-many elements in $\{x_i\}_i$. Moreover, one can check that $\bigcup_{i\leq e^{n(h+\Delta+\delta)}} B(x_i,n,\rho)\supseteq K_\delta$.
\end{proof}

\begin{cor}\label{forC1Thm}
Let $\chi_0>0$, $\tau\in (0,\frac{\chi_0}{100d})$, $\epsilon\geq 4\tau$ small (w.r.t $\chi_0$), $\gamma\geq \sqrt\epsilon$, $\ell\in\mathbb{N}$, and let $n'$ s.t $\mu(K)>0$ where
\begin{align}\label{adon}
K\subseteq\Big\{x\text{ Lyap. reg. s.t }\chi^u_{\min}(x),\chi^s_{\min}(x)\geq \chi_0: &\\
\forall n\geq n',& x\in\Lambda_\ell^{(\underline{\chi}(x),\tau)}\cap\Big(\bigcup_{n(1+\gamma)\leq i\leq n(1+2\gamma)} f^{-i}[\Lambda_\ell^{(\underline{\chi}(x),\tau)}]\Big)\Big\}.\nonumber
\end{align}
Assume that $h_\mu^\mathrm{BK}\leq h+\Delta$ for $\mu$-a.e $x\in K$. Then there exists $K_\tau\subseteq K$ s.t $\mu(K_\tau)\geq e^{-\tau}\mu(K)$ and for all $n$ large enough, $\exists \mathcal{A}_n\subseteq K_\tau$ s.t $\bigcup_{x\in\mathcal{A}}B(x,n,e^{-\epsilon n})\supseteq K_\tau$ and $\#\mathcal{A}_n\leq e^{n(1+2\gamma)(h+\Delta+3d\epsilon)}$.
\end{cor}
\begin{proof}
Let $K_%{\frac{\tau}{2}}'
\tau$ and $n_\tau$ as in Lemma \ref{hBddBowenBallsCover}, for some $0<\rho\leq\tau$. 

%Write $K_{\frac{\tau}{2}}'=\bigcup_{i=1}^{\frac{2d\log M_f}{\tau}}K_{\frac{\tau}{2}}^{',i}$, where $K_\frac{\tau}{2}^{',i}:=\{x\in K_{\frac{\tau}{2}}': \sum \chi^+(x)=i\frac{\tau}{2}\pm \frac{\tau}{2}\}$. Then let $K_\tau^i\subseteq K_{\frac{\tau}{2}}^{',i}$ be a set of Lebesgue density points of $K_\frac{\tau}{2}^{',i}$ s.t $\forall x\in K_\tau^i$ and $r\in (0,r_i)$, we have $\mu(B(x,r)\cap K_\frac{\tau}{2}^{',i})\geq e^{-\frac{\tau}{2}}\mu(B(x,r))$ and $\mu(K_\tau^i)\geq e^{-\frac{\tau}{2}}\mu(K_\frac{\tau}{2}^{',i})$. Set $K_\tau:=\bigcupdot K_\tau^i$, and notice that $\mu(K_\tau)\geq e^{-\tau}\mu(K)$.

Let $n\geq  \max\{n_\tau,n'\}$% large enough so $e^{-\epsilon n}\leq \min_i r_i$
, and let $\mathcal{C}_{n}$ be a Besicovitch cover of $K_\tau$ by balls of radius $e^{-2\epsilon n}$. 

For each such ball $B%=B(x_B, e^{-2\epsilon n})$, where $x_B\in K_\tau^{i_B}
$, we cover $K_%\frac{\tau}{2}^{',i_B}
\tau\cap B$ with at most $e^{n(1+2\gamma)(h+\Delta+\tau)}$-many Bowen balls of the form $B(\cdot, \lfloor n(1+2\gamma)\rfloor,\rho)$ by Lemma \ref{hBddBowenBallsCover}. Hence in total we cover $K_\tau$ with at most $B_dC_Me^{2d\epsilon n}e^{n(1+2\gamma)(h+\Delta+\tau)}$-many elements, where $B_d$ is the Besicovitch constant of $M$, and $C_M$ is a constant s.t $\Vol(B(x,e^{-2\epsilon n}))\geq \frac{1}{C_M}e^{-2d\epsilon n}$ for all $x\in M$ and $n$ large.

Let $B\in \mathcal{C}_n$, and let $\mathcal{A}^B_{n,\tau}$ as in Lemma \ref{hBddBowenBallsCover} for $B\cap K_\tau$. Let $x\in \mathcal{A}^B_{n,\tau}$, and notice that for $j_x\in[\lfloor n(1+\gamma)\rfloor, \lfloor n(1+2\gamma)\rfloor]$,
\begin{equation}\label{liliment}
B(x, \lfloor n(1+2\gamma)\rfloor,\rho)\cap B \subseteq B(x, j_x,\rho)\cap B(x,2e^{-2\epsilon n})\subseteq B(x,n,e^{-\epsilon n}),
\end{equation}
 for all $n$ large enough w.r.t $\ell$ and $\epsilon$ (we prove \eqref{liliment} in the end of this lemma). Thus in total, $\{B(x,n,e^{-\epsilon n}): x\in \mathcal{A}^B_{n,\tau},B\in\mathcal{C}_n\}$ is  a cover of $K_\tau$ by exponential Bowen balls, of cardinality bounded by $e^{n(1+2\gamma)(h+\Delta+3d\epsilon)}$ for all $n$ large enough.
 
\medskip To prove \eqref{liliment}, we work with Pesin charts, which is where we need the assumption from \eqref{adon}. %We may assume that $n_\tau$ is large enough w.r.t $\ell$, and that $\rho$ is small w.r.t $\ell$. 

Let $x\in K$. We wish to show that for all $n$ large enough (w.r.t on $\ell$), $\forall i\in [1,n]$, $\forall y\in B(x,j_x,\rho)$, $f^i(y)\in B(f^i(x),e^{-\epsilon n})$. Letting $\psi_i$ be the Pesin chart of $f^i(x)$, it is enough to show that $|\psi_i^{-1}\circ f^i\circ \psi_0(v_y)|\leq e^{-\frac{5}{4}\epsilon n}$, where $v_y:=\psi_0^{-1}(y)$. 

Write $\psi_i^{-1}\circ f^i\circ \psi_0(v_y)=v=v^s+v^c+v^u$, where $v^t\in E^t(x)$, $t\in\{s,c,u\}$. Set $f_i:=\psi_{i+1}^{-1}\circ f\circ \psi_i$, and $F_i:=f_{i-1}\circ\cdots \circ f_0$. We assume for the simplicity of presentation that all of the negative Lyapunov exponents of $x$ are equal, and that all of the positive Lyapunov exponents of $x$ are equal, otherwise decompose $v_s$ and $v_u$ into corresponding components.

A standard result of Pesin theory tells us that the maps $f_i$ can be put in the form $f_i=\sum_{t\in  \{s,c,u\}}D_t v_t+h_i^t(v)$, $\| h_i^t\|_{C^1}\leq \tau$, and where $D_t$ are linear self-maps of $E^t$, and $ e^{\chi^s(x)+\tau}\geq \|D_s^{-1}\| ,\|D_s\|\leq e^{-\chi^s(x)+\tau}$, $ e^{-\chi^u(x)+\tau}\geq \|D_u^{-1}\| ,\|D_u\|\leq e^{\chi^u(x)+\tau}$, and $\|D_c^{-1}\|, \|D_c\|\leq e^{\tau} $.

Therefore the stable and central components of $F_n(v_y)$ remain small enough, and we are left to bound $v^u=(F_n(v_y))^u$. Since $f^{j_x}(y)\in B(f^j(x),\rho)$, similar contraction estimates hold for $f^{-1}$, and we get $|v_y^u|\leq e^{-(\chi^u(x)-2\tau)n(1+\gamma)}$. Thus, $| (F_n(v))^u |\leq e^{-(\chi^u(x)-2\tau)n(1+\gamma)} e^{(\chi^u(x)+2\tau)n}\leq e^{-\gamma n(\chi_0-4\tau)}\leq e^{-\frac{5}{4}\epsilon n}$ for $\epsilon>0$ small enough (w.r.t $\chi_0$). 
\end{proof}

\begin{theorem}\label{CoolestFormulaForC1Maybe} Let $f\in\mathrm{Diff}^{1+\beta}(M)$, where $M$ is a closed Riemannian manifold with $\mathrm{dim}M=d\geq2$. Let $\mu$ be an $f$-invariant Borel probability on $M$. Then 
	$$\mathcal{E}_\mu(x)=h_\mu^\mathrm{BK}(x)\text{ }\mu\text{-a.e.}$$
\end{theorem}
\begin{proof}
We start with a reduction. Let $\mu=a^+\mu^++a^0\mu^0$ where $\mu^+$ admits a positive Lyapunov exponent a.e, and $\mu^0$ has all exponent less or equal to $0$ a.e. Then $\E_\mu\leq\E_{\mu^+}$, and for $\mu^0$-a.e $x$, $\E_\mu(x)=0= \h_\mu(x)$ (by dimension bounds). Therefore we may assume w.l.o.g that $\mu$ admits a positive Lyapunov exponent a.e. Moreover, write $\mu=\sum_{i\geq 1}a_i\mu_i$ where for every $i$, almost every exponent of $\mu_i$ is greater than $\frac{1}{i}$ in absolute value, then for every $i$, $\mu_i$-a.e, by Lemma \ref{forReduc}, $$\h_\mu= \h_{\mu_i} \overset{?}{\geq}\E_{\mu_i}\geq \E_\mu.$$
Then it is enough to assume that the non-zero Lyapunov exponents of $\mu$ are uniformly bounded from below in absolute value by a constant $\chi_0>0$.

It is enough to assume for contradiction that $\exists \lambda\in(0,\min\{\chi^3,\frac{1}{2}\})$ s.t $\E_\mu\geq \h_\mu+2\lambda$ $\mu$-a.e, since if $\mu=a_\lambda\mu_\lambda+(1-a_\lambda)\mu_{\lambda}^-$ where $a_\lambda>0$ and $\E_\mu\geq h^\mathrm{LY}_\mu+2\lambda$ $\mu_\lambda$-a.e , then once more by Lemma \ref{forReduc},
$$\E_{\mu_\lambda}\geq \E_\mu\geq \h_\mu+2\lambda= \h_{\mu_\lambda}+2\lambda,\text{ }\mu_\lambda\text{-a.e.}$$
Finally, we make the following reduction: let $G_{(\E,h)}:=\{x: \E_\mu(x)=\E\pm \frac{\lambda}{2}, \h_\mu(x)=h\pm \frac{\lambda}{2}\}$, then let $\mu=a_{(\E,h)}\mu'+(\mu-a_{(\E,h)}\mu')$, where $a_{(\E,h)}>0$ and $\mu'$ is carried by $G_{(\E,h)}$. Then, $\mu'$-a.e
\begin{equation}\label{sundayTrail}
	\E_{\mu'}\geq \E_\mu \geq  \h_\mu+2\lambda= \h_{\mu'}+2\lambda.
\end{equation}
Moreover, similarly, we may assume w.l.o.g that $\E_\mu\gg \sqrt{\lambda}$ a.e. Therefore we may assume for contradiction that $\E_\mu$ and $\h_\mu$ are ``almost constant" w.r.t to the gap between them which is uniformly bounded from below.

\medskip
There exists $0<r\leq \frac{\min\{1,\chi_0\}}{3d}\frac{\lambda^4}{4}$ s.t $\mu(A_1)\geq 1-\lambda^3$ where $A_1:=\{x\in A:\limsup\frac{-1}{n}\log\mu(B(x,n,e^{-rn}))>h_\mu^\mathrm{BK}(x)+\frac{7\lambda}{8}\}$.

We can then choose $0<\tau\leq\min\{\frac{r}{400}\}$ and $\ell\in\mathbb{N}$ s.t $\mu(A_2)\geq 1-2\lambda^4$ where $$A_2:=\{x\in A_1\text{ Lyapunov reg.}: x\in \Lambda_\ell^{(\underline{\chi}(x),\tau)}\}.$$ %Finally, $\exists h>0$ s.t $\mu(K)>0$ where $K:=\{x\in A_2: h_\mu^\mathrm{BK}(x)=h\pm \frac{\lambda}{8}\}$.

Let $\mu=\int \mu_x d\mu(x)$ be the ergodic decomposition of $\mu$. Then by the Markov inequality, $\mu(\{x:\mu_x(A_2)\geq 1-\sqrt{2\lambda^4}\})\geq 1-\sqrt{2\lambda^4}$. Then $\mu(\{x\in A_2: \mu_x(A_2)\geq 1-\sqrt{2}\lambda^2\})\geq 1-2\lambda^2$. So by the ergodic theorem $\exists n_0\in\mathbb{N}$ s.t $\mu(A_3)\geq 1-3\lambda^2$ where $$A_3:=\left\{x\in A_2: \forall n\geq n_0 \exists j\in[n(1+4\lambda^2),n(1+8\lambda^2)]\text{ s.t }f^j(x)\in A_2 \right\}.$$

Let $K_\tau\subseteq A_3$ as in Corollary \ref{forC1Thm} s.t $\mu(K_\tau)>0$ (notice, the restrictions on $K_\tau$ in Corollary \ref{forC1Thm} are given by Lemma \ref{hBddBowenBallsCover}, which in turn are merely Brin-Katok estimates; which are inherited by subsets). Let $N\geq n_0$ large, then for all $x\in K_\tau$ set $$n_x^N:=\min\{n\geq N: \frac{-1}{n}\log\mu(B(x,n,e^{-rn}))>h+\frac{6\lambda}{8}\}.$$
For all $n\geq N$, set $K_n:=\{x\in K_\tau:n_x^N=n\}$.

By Corollary \ref{forC1Thm}, we can cover $K_n$ with a cover whose cardinality is less or equal to $e^{n(1+ 8\lambda^2)(h+ \frac{\lambda}{8}+3dr)}$, of exponential Bowen balls of the form $B(x,n,e^{-nr})$, $x\in K_n$. Hence, $\mu(K_n)\leq e^{n(1+8\lambda^2)(h+\frac{\lambda}{8}+3dr)} \cdot e^{-n(h+\frac{6\lambda}{8})}$, whence $0<\mu(K_\tau)\leq \sum_{n\geq N} e^{-n\frac{\lambda}{8}}\xrightarrow[]{N\to\infty}0$, a contradiction! Hence $h_\mu^\mathrm{BK}\leq \mathcal{E}_\mu\leq h_\mu^\mathrm{BK}$ $\mu$-a.e.
\end{proof}

\noindent\textbf{Remark:} Theorem \ref{CoolestFormulaForC1Maybe} implies that $\mathcal{E}_\mu(x)=\lim_{r\to0}\liminf_{n\to\infty}\frac{-1}{n}\log\mu(B(x,n,e^{-rn}))$ for $\mu$-a.e $x$ since $$h_{\mu}^{\mathrm{BK}}(x)= \lim_{r\to0}\liminf_{n\to\infty}\frac{-1}{n}\log\mu(B(x,n,r)) \leq \lim_{r\to0}\liminf_{n\to\infty}\frac{-1}{n}\log\mu(B(x,n,e^{-rn})) \leq \mathcal{E}_\mu(x)=h_{\mu}^\mathrm{BK}(x).$$

\begin{comment}
\begin{cor}
Let $\mu$ be an $f$-invariant Borel probability, and let $\mu=\int \mu_x d\mu(x)$ be its ergodic decomposition. Then $\E_\mu(x)= h_{\mu_x}(f)$ $\mu$-a.e. %Assume that $\exists h,sigma>0$ s.t for $\mu$-a.e $x$, $h_{\mu_x}(f)=h+\pm\sigma$. Then $\E_\mu=h\pm \sigma$ $\mu$-a.e.
\end{cor}
\begin{proof}
	Let $E_\lambda:=\{x: \E_\mu(x)\geq h_{\mu_x}(f)+\lambda\}$, and assume that there exists $\lambda>0$ s.t $\mu(E_\lambda)>0$ (notice that $E_\lambda$ is $f$-invariant). Let $G_\lambda:=\{x: \mu_x(E_\lambda)=1\}$ with $a_\lambda:=\mu(G_\lambda)>0$ (o.w $\mu(E_\lambda)=0$). Write $\mu_\lambda:=\frac{1}{\mu(G_\lambda)}\int_{G_\lambda}\mu_x d\mu(x)$, and $\mu_c:=\frac{1}{a_c}(1-\mu_\lambda)$ where $a_c:=1-a_\lambda$. Then $\mu=a_\lambda \mu_\lambda+a_c\mu_c$ and for $\mu_\lambda$-a.e $x$,%{\color{red} needs to be explained why $\mu_\lambda$ is carried by $E_\lambda$ (notice $E_\lambda$ is invariant)}
	
$$\E_{\mu_\lambda}(x)\geq \E_\mu(x)\geq h_{\mu_x}(f)+\lambda.$$
However, this is a contradiction, since integrating both sides by $\mu_\lambda$ admits $h_{\mu_\lambda}(f)\geq h_{\mu_\lambda}(f)+\lambda$, a contradiction! Hence $\E_\mu(x)\leq h_{\mu_x}(f)$ $\mu$-a.e, but $\int \E_\mu(x)d\mu(x)=\int h_{\mu_x}(f)d\mu(x)$, hence $\E_\mu(x)=h_{\mu_x}(f)$ $\mu$-a.e.	
\end{proof}
\end{comment}

\section{Lower-dimension bounds for invariant measures}\label{LowDim}

In \cite{LedrappierYoungII}, the authors consider an  ergodic $f$-invariant measure $\mu$, where $f$ is a $C^2$ diffeomorphism of a closed Riemannian manifold $M$. The authors consider two important partitions- $\xi^u$ and $\xi^s$ sub-ordinated to the unstable and to the stable laminations respectively, and consider the conditional measures of $\mu$ w.r.t to these partitions. The authors then go on and prove that the conditional measures $\mu_{\xi^u(\cdot)}$ and $\mu_{\xi^s(\cdot)}$ are exact-dimensional for almost every point, with the dimensions being constant and denoted by $d^u$ and $d^s$ respectively. Moreover, the authors prove that the point-wise upper-dimension of $\mu$-a.e point is bounded by $$\overline{d}\leq d^u+d^s+d^c,$$
where $d^c=d^c(\mu):=\#\{0\text{ Lyapunov exponent of }\mu\}$.

In \cite{AsymProd}, the authors extend the results of \cite{LedrappierYoungII} by relaxing the regularity assumption of $f$ to $C^{1+\beta}$ ($\beta>0$) by proving the Lipschitz property of intermediate foliations; and bound from below the point-wise dimension of ergodic invariant measures under the additional assumption of hyperbolicity:
$$d^u+d^s\leq \underline{d}\leq \overline{d}.$$
In particular, it follows as a consequence that a hyperbolic measure is exact-dimensional, since $d^c=0$.

The purpose of this section, is to extend these results, as an application of Theorem \ref{CoolestFormulaForC1Maybe}. We give lower bounds to the point-wise dimension of invariant measures, which coincide with the lower bounds from \cite{AsymProd} when the measure is hyperbolic. In general, invariant measures with $0$ Lyapunov exponents need not be exact-dimensional, and the bounds from \cite{LedrappierYoungII} are tight. In particular, our proof is aimed to be short and accessible by using the neutralized local entropy instead of adapted partitions.

\begin{theorem}\label{LowDimBound}
	Let $\mu$ be an $f$-invariant ergodic measure, where $f\in \mathrm{Diff}^{1+\beta}(M)$ ($\beta>0$ and $M$ a closed Riemannian manifold). Then for $\mu$-a.e $x$,
	$$d^s+d^u\leq \underline{d}(x):=\liminf_{r\to0}\frac{\log\mu(B(x,r))}{\log r}.$$
\end{theorem}
\begin{proof}
		We start with two simple reductions. Given $r>0$, and $\chi>0$, let $r':=\max\{e^{-n\chi}: e^{-n\chi}\leq r, n\in\mathbb{N}\}$. Then,
\begin{align*}
			\frac{\log \mu(B(x,r))}{\log r} =& \frac{\log r'}{\log r} \cdot \frac{\log \mu(B(x,r))}{\log \mu(B(x,r'))} \cdot \frac{\log \mu(B(x,r'))}{\log r'}\\
		\geq &(1-\frac{\chi}{\log r})\cdot \frac{\log \mu(B(x,r'))}{\log r'}.
\end{align*}
Then it is enough to prove $\liminf_{n\to\infty}\frac{\log\mu(B(x,e^{-n\chi}))}{-n\chi}\geq d^u+d^s$, for some $\chi>0$.
The second assumption we make is that $\mu$ admits some non-zero Lyapunov, since otherwise $d^s=d^u=0$, and the statement is trivial. Set $\chi:=\frac{1}{3}\min\{|\chi_i(\mu)|:\chi_i(\mu)\neq 0\}>0$.

\medskip
We now can start with the construction for the proof. Assume that $\mu$ admits positive exponents, and so we prove $\underline{d}\geq d^u$. If $\mu$ has no negative exponents, this completes the proof; the case of negative exponents is treated subsequently. 

Let $\delta\in (0,\chi)$, and let $\epsilon\in(0,\frac{\delta}{3})$ and $n_\delta\in\mathbb{N}$ s.t $e^{-\chi n_\delta}\ll\frac{1}{\ell_\delta^3}$ and $\mu(A_\delta)\geq 1-\delta$ where
\begin{align*}
	A_\delta:=\{x\in \Lambda^{(\underline\chi,\tau)}_{\ell_\delta}:\forall n\geq n_\delta,& \mu_{\xi^u(x)}(B^u(x,e^{-\chi n}))\leq e^{-\chi d^u n+n\delta},\\
	 &\mu(B(x,-n,e^{-\epsilon n}))= e^{-n \E_\mu\pm\delta n},\mu(B(x,n,e^{-\epsilon n}))= e^{-n \E_\mu\pm\delta n},\\
	  &\mu(B(x,-n,n,2e^{-\epsilon n}))=e^{-2\E_\mu n\pm \delta n}\},
\end{align*}
where $\underline\chi=\underline{\chi}(\mu)$ and $\tau\in (0,\frac{\epsilon}{100d})$, and $B(\cdot, -n,e^{-\epsilon n})$ denotes an exponential Bowen balls for $f^{-1}$, while $B(\cdot, -n,n,2e^{-\epsilon n})$ denotes a two-sided exponential Bowen ball. 

By Lemma \ref{LLYD}, there exists $A_\delta'\subseteq A_\delta$ and $m_\delta\geq n_\delta$ s.t $\mu(A_\delta')\geq 1-2\delta$ where\footnote{While formally Lemma \ref{BLDT} is only stated for the limit $\lim_{r\to0}\limsup_{n\to\infty}\frac{-1}{n}\log\frac{\mu(B(x,n,e^{-nr})\cap A)}{\mu(B(x,n,e^{-nr}))}=0$, the quantitative argument extends as is to the quantitative estimate of \eqref{QuantDens} whenever $\epsilon>0$ is small w.r.t $\chi$.} \begin{equation}\label{QuantDens}
A_\delta':=\{x\in A_\delta:\forall n\geq m_\delta,\frac{1}{n}\log\frac {\mu(B(x,n,e^{-\epsilon n}))}{\mu(B(x,n,e^{-\epsilon n}) \cap A_\delta)} \leq 48d\epsilon\}. 	
\end{equation}

Let $x\in A_\delta'$ which is a Lebesgue density point s.t $\mu(B(x,e^{-n\chi})\cap A_\delta')\geq e^{-\delta} \mu(B(x,e^{-n\chi}))$ for all $n\geq n_x\geq m_\delta$, and let $n\geq n_x$.

Then cover $A_\delta'\cap B(x,e^{-\chi n})$ as in Lemma \ref{BBCL}, by a cover $C^u$ of balls $B(\cdot, n, e^{-\epsilon n})$ with multiplicity bounded by $e^{3d\tau n}$. 

It follows that, 
$$\mu(\bigcup C^u\cap A_\delta)\geq e^{-3d\tau n}\cdot \#C^u\cdot \min_{B\in C^u}\mu(B\cap A_\delta)\geq e^{-3d\tau n}\cdot \#C^u\cdot e^{-n \E_\mu-\delta n}\cdot e^{-48d\epsilon n}.$$
Then, there exists $x'$ s.t $\mu_{\xi^u(x')}(\bigcup C^u\cap A_\delta)\geq \mu(\bigcup C^u\cap A_\delta)$ and so $$\mu_{\xi^u(x')}(\bigcup C^u)\geq \mu_{\xi^u(x')}(\bigcup C^u\cap A_\delta)\geq \mu(\bigcup C^u\cap A_\delta)\geq e^{-3d\tau n}\cdot \#C^u\cdot e^{-n \E_\mu-\delta n-48\epsilon d n},$$ where we may assume w.l.o.g that $x'\in \bigcup C^u\cap A_\delta$.

\medskip
\noindent\underline{Claim:} $\xi^u(x')\cap \bigcup C^u\subseteq B^u(x',e^{-n\chi+2n\epsilon})$.

\noindent\underline{Proof:} Since $\xi^u(x')\cap B(x',3e^{-\epsilon n})\supseteq \bigcup C^u\cap \xi^u(x')$ and is contained in the Pesin chart of $x'$ (recall $x,x'\in \Lambda_{\ell_\delta}^{(\underline\chi,\tau)}$), we get that $d^u(x',z)$ is up to $e^{\pm n\frac{\epsilon}{3}}$, $|x'_u-z_u|$, where $\cdot_u$ denotes the $E^u(x)$ component in the Pesin chart of $x$.

To bound $|x'_u-z_u|$, where $z\in \xi^u(x')\cap\bigcup C^u$, it is enough to bound $|y'_u-y_u|$ for any $y\in A_\delta\cap B(x,e^{-n\chi})$ and $y'\in B(y,n,e^{-\epsilon n})$, by the triangle inequality.

Indeed, $|y'_u-y_u|\leq e^{-n\chi}$ by the chart estimates to satisfy the Bowen ball condition. QED

\medskip
Therefore by the claim, we get in total, $$e^{-(\chi-2\epsilon)n(d^u-\delta)}\geq \mu_{\xi^u(x')}(B^u(x',e^{-\chi n+2\epsilon n}))\geq e^{-3d\tau n}\cdot \#C^u\cdot e^{-n \E_\mu-\delta n-48\epsilon d n}.$$
Thus,
\begin{equation}\label{CuUpperBound}
	\# C^u\leq e^{n(\E_\mu-d^u\chi+6\delta)}.
\end{equation}
This concludes the bound $$\mu(B(x,e^{-\chi n}))\leq e^\delta\cdot \mu(B(x,e^{-\chi n})\cap A_\delta')\leq e^\delta \cdot \#C^u\cdot e^{-n(\E_\mu+\delta)}\leq e^\delta\cdot e^{n(\E_\mu-d^u\chi+6\delta)}\cdot  e^{-n(\E_\mu+\delta)}.$$ Hence $\underline{d}\geq d^u-7\delta$ where $\delta>0$ is arbitrary.

For the case where $\mu$ admits negative exponents as well, construct similarly a cover $C^s$ of $A_\delta'\cap B(x,e^{-\chi n})$ by exponential Bowen balls for $f^{-1}$, denoted by $B(\cdot, -n,e^{-\epsilon n})$. Then similarly one gets that $\#C^s\leq e^{n(\E_\mu-d^s\chi+6\delta)}$. We then define the cover $C:=\{B^s\cap B^u: B^s\in C^s, B^u\in C^u,\text{ and } B^s\cap B^u\cap A_\delta'\neq \varnothing\}$. 

It follows immediately that $\# C\leq \# C^u\cdot \#C^s\leq e^{n(2\E_\mu-d^s\chi-d^u\chi+12\delta)}$. Moreover, for every element $B\in C$, $B\subseteq B(x_B,-n,n,2e^{-\epsilon n})$ where $x_B\in A_\delta'$ by the triangle inequality. Therefore,
$$\mu(B(x,e^{-\chi n}))\leq e^\delta\cdot \# C\cdot e^{-n(2\E_\mu-\delta)}\leq e^\delta\cdot e^{-n(d^s+d^u)\chi+13\delta n}.$$
Since $\delta>0$ was arbitrary, we are done. 
\end{proof}

\noindent\textbf{Remark:} Note that the estimate from above of the measure of a ball in Theorem \ref{LowDimBound} is quite coarse: All elements of $C^u$ and of $C^s$ are of size $\sim e^{-\epsilon n}$ in the central direction, which is much longer than the diameter of the ball $e^{-\chi n}$. While this over-shooting may seem wasteful, one may may not expect an invariant measure to be concentrated in the central-direction, and in fact generally an invariant measure may even be atomic in the central direction, thus the estimate is tight for the general case. In some cases where we have more information regarding the central direction, we may say a bit more, as shown in Corollary \ref{CLeaves}. 
	
\begin{cor}\label{CLeaves}
Under the assumptions of Theorem \ref{LowDimBound}, if $\mu$ admits a measurable lamination by central leaves almost everywhere, then a.e, $$\limsup_{r\to 0}\frac{\log \mu(B(x,r))}{\log r}:=\overline{d} (\mu)=\overline{d}\geq d^s+d^u+\underline{d}^c,$$
where $\underline{d}^c$ is the lower point-wise dimension of the conditional measures on the lamination by central leaves. 
\end{cor}
%For Snir: given any two partition into central leaves, they can be refined into a new one, as long as the elements are made up of W^c-open sets. Then the notion of central dimension is well-defined. Moreover, it is invariant by refining \eta^c:=\xi^c \wedge f^{-1}\xi^c, and the formula \mu_{\xi^c(x)}=\sum_{\eta^c(y)\subseteq \xi^c(x)}\mu_{\xi^x(x)}(\eta^c(y))\cdot \mu_{\eta^c(y)}. 
\begin{proof}
	Let $\xi^c$ be the measurable partition of $\mu$ into central leaves. Notice that $\underline{d}^c$ and $\overline{d}$ are invariant functions, and are constant a.e. %Assume for contradiction that $\exists \lambda\in(0,1)$ s.t $\underline{d}^c>\lambda+\underline{d}-d^u-d^s$, and l
	Let $\delta>0%\in(0,\frac{\lambda}{100})
	$. We impose the following additional restrictions on the set $A_\delta$ from Theorem \ref{LowDimBound}: $x\in A_\delta\Rightarrow \forall n\geq n_\delta$, $\mu_{\xi^c(x)}(B^c(x,e^{-\chi n}))\leq e^{-\chi \underline{d}^c n+\delta n}$, and $\mu(B(x,e^{-\chi n}))\geq e^{-n\chi \overline{d}-\delta n}$
	.
	
We consider the set $A_\delta''$, which is the set of density points of $A_\delta'$ with dimension bounds. So for $x\in A_\delta''$, and all $n\geq n_\delta'\geq n_\delta$, $$\mu(B(x,e^{-2\epsilon n})\cap A_\delta')\geq e^{-3\epsilon d n}.$$

Then there exists $x''\in B(x,e^{-2\epsilon n})\cap A_\delta'$ s.t $\mu_{\xi^c(x'')}(B(x'',e^{-\epsilon n})\cap A_\delta')\geq e^{-3d\epsilon n}$.

Consider the ``tube" $T_n(x''):=\psi_{x''}[R^s(0,e^{-\chi n})\times R^c(0,e^{-\epsilon n})\times R^u(0,e^{-\chi n})]$, then $\mu_{\xi^c(x'')}(T_n(x'')\cap A_\delta')\geq e^{-3\epsilon d n }$. We consider a Besicovitch cover of $T_n(x'')\cap A_\delta'$ by balls of radius $e^{-\chi n+\epsilon n}$, $C^c$, hence $\#C^c\geq e^{(\chi-\epsilon) n \underline{d}^c-2\delta n -4\epsilon d n}$ for all $n$ large enough.

On the other hand, Theorem \ref{LowDimBound} gives the upper bound $\mu(T_n(x''))\leq e^{-\chi(d^s+d^u)n+13\delta n}$ (recall the remark after Theorem \ref{LowDimBound}). Then, %if we knew that 
since $\mu(B(z,e^{-\chi n}))\geq e^{-\underline{d}\chi n-\delta n}$ with $z\in A_\delta$, we have
$$e^{-\overline{d}\chi n-\delta n}\leq \frac{\mu(T_n(x''))}{\# C^c}\leq \frac{e^{-\chi(d^s+d^u)n+13\delta n}}{e^{(\chi-\epsilon) n \underline{d}^c-2\delta n -4\epsilon d n}},$$
which concludes the proof since $\delta>0$ is arbitrary. %would conclude the proof since $\delta>0$ can be arbitrarily small. However, the lower point-wise dimension is an upper bound on the measure of balls. We treat this issue below.
%
%Let $N\geq n_\delta'$, and let $E_n^N:=\{x\in A_\delta'': n=\min\{m\geq N: \mu(B(x,e^{-\chi m}))\geq e^{-m\chi \underline{d}-\delta m} \}\}$. Then $A_\delta''=\bigcupdot_{n\geq N}E_n^N$. For $n\geq N$, cover $E_n^N$ by tubes $T_n(\cdot)$ with a cover $C_n$ of cardinality bounded by $e^{2\epsilon n}B_d^2$, where $B_d$ is the Besicovitch constant of $M$; such a cover is achieved as following: first we cover $E_n^N$ by balls of radius $\frac{1}{2}e^{\epsilon n}$, with a multiplicity bound of $B_d$. For each such ball, we cover $E_n^N$ inside of it with a tubes, which is effectively a cover by balls of radius$=e^{-\chi n\pm\epsilon n}$ in the $s$-$u$ directions in the chart, since the central direction of the tubes spans across the ball.
\end{proof}

\bibliographystyle{alpha}
\bibliography{Elphi}

\end{document}